\documentclass[10pt,reqno]{amsart}
\usepackage{amsmath}
\usepackage{amsthm}
\usepackage{amssymb}
\usepackage{graphicx}
\usepackage{amsfonts}
\usepackage{latexsym}
\usepackage{setspace}
\usepackage{hyperref,caption,subcaption}
\usepackage{cite}
\usepackage{multicol}
\usepackage[active]{srcltx}
\usepackage{mathrsfs}

\newcommand{\C}{ \mathbb{C}}
\newcommand{\Z}{ \mathbb{Z}}
\renewcommand{\pmod}[1]{\,(\operatorname{mod}#1)}

\theoremstyle{plain}
\newtheorem*{Thm}{Theorem}

\allowdisplaybreaks

\begin{document}

    \title{Quotients of Gaussian primes}

    \author{Stephan Ramon Garcia}
    \address{   Department of Mathematics\\
            Pomona College\\
            Claremont, California\\
            91711 \\ USA}
    \email{Stephan.Garcia@pomona.edu}
    \urladdr{\url{http://pages.pomona.edu/~sg064747}}
    
    \begin{abstract}
    It has been observed many times, both in the \textsc{Monthly} and elsewhere, 
    that the set of all quotients of prime numbers is dense in the positive real numbers. In this short note we answer the related question: 
    ``Is the set of all quotients of Gaussian primes dense in the complex plane?''
    \end{abstract}

\maketitle
	  
Quotient sets $\{ s/t :s,t \in \mathbb{S}\}$ corresponding to subsets $\mathbb{S}$ of the natural numbers
have been intensely studied in the \textsc{Monthly} over the years \cite{Bukor,Hedman,Hobby,Nowicki,Starni, QSDE}. 
Moreover, it has been observed many times in the \textsc{Monthly} and elsewhere that the set of all quotients of prime numbers is dense in the positive reals (e.g., \cite[Ex.~218]{DM}, \cite[Ex.~4.19]{FR}, 
\cite[Thm.~4]{Hobby}, \cite[Cor.~5]{QSDE}, \cite[Ex.~7, p.~107]{Pollack}, \cite[Thm.~4]{Ribenboim}, \cite[Cor.~2]{Starni}).  

In this short note we answer the related question: \emph{``Is the set of all quotients of Gaussian primes
dense in the complex plane?''}  The author became convinced of the nontriviality of this problem
after consulting several respected number theorists who each admitted not seeing a simple solution.  

In the following, we refer to the traditional primes $2,3,5,7,\ldots$ as \emph{rational primes}, remarking that a rational prime $p$ is a \emph{Gaussian prime} (i.e., a prime in the ring $\Z[i] := \{a + bi : a,b \in \Z\}$ of \emph{Gaussian integers}) if and only if $p \equiv 3 \pmod{4}$.  In general, a nonzero Gaussian integer is prime if and only if it is of the form $\pm p$ or $\pm p i$ where $p$ is a rational prime congruent to $3\pmod{4}$ or if it is of the form $a+bi$ where $a^2 +b^2$ is a
rational prime (see Figure \ref{FigureMain}).
We refer the reader to \cite{HardyWright} for complete details.  
\begin{figure}
	\begin{center}
		\includegraphics[width=2.25in]{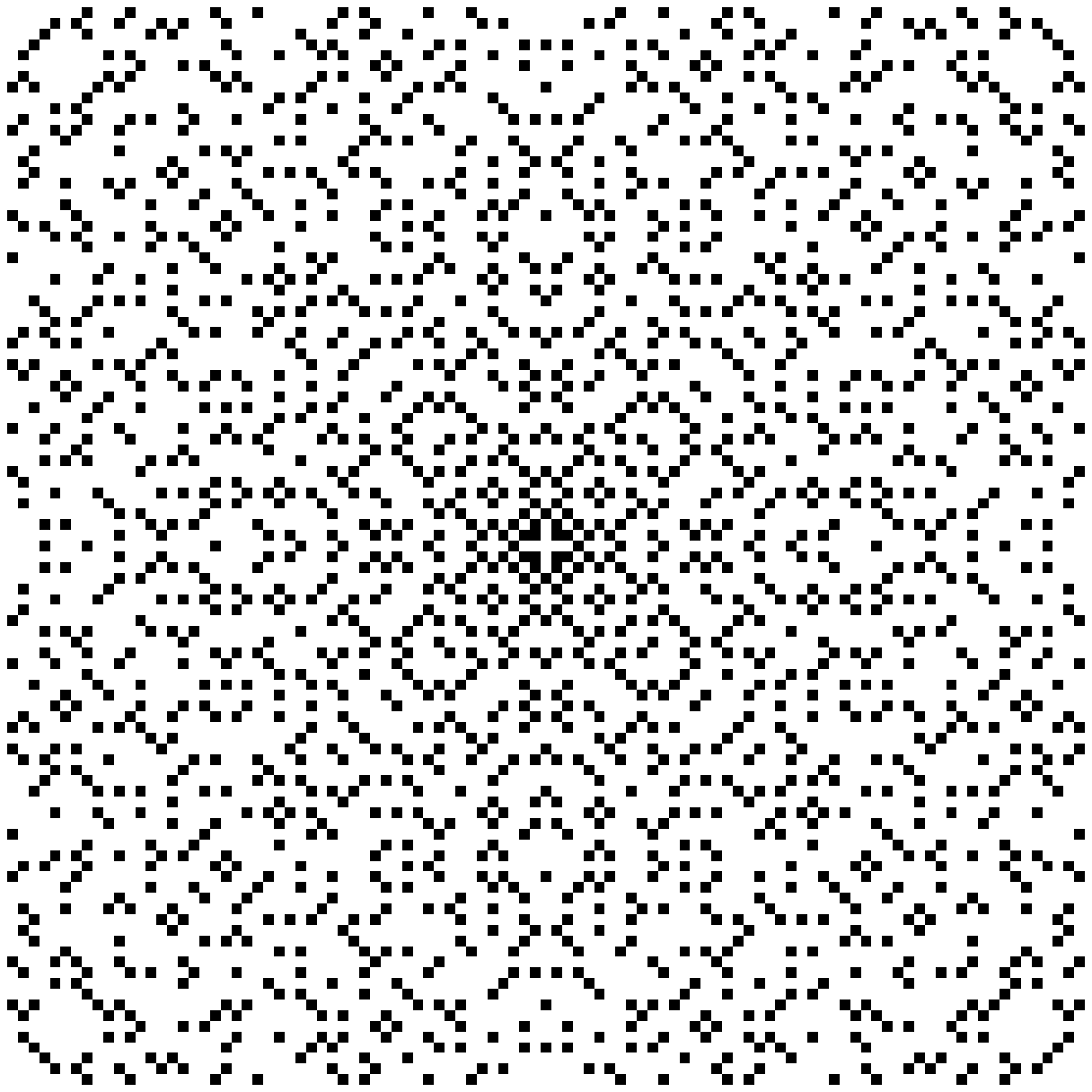}\qquad
		\includegraphics[width=2.25in]{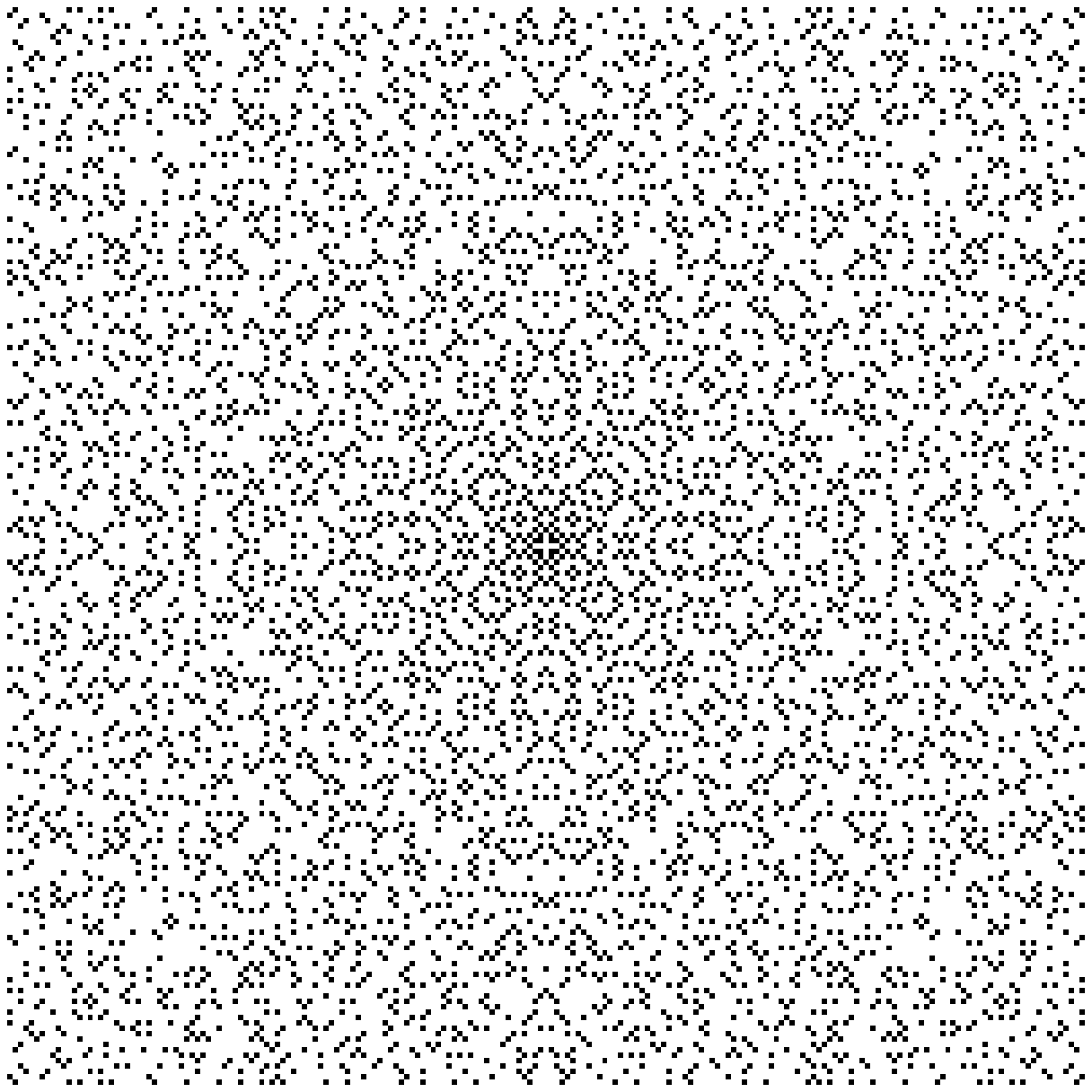}
	\end{center}
	\caption{Gaussian primes $a+bi$ satisfying $|a|,|b| \leq 50$ and $|a|,|b|\leq 100$, respectively.}
	\label{FigureMain}
\end{figure}	

\begin{Thm}
	The set of quotients of Gaussian primes is dense in the complex plane.
\end{Thm}

\begin{proof}
	It suffices to show that each region of the form
	\begin{equation}\label{eq-Region}
		 \{ z \in \C : \alpha < \arg z < \beta, \,\, r < |z| < R \}.
	\end{equation}
	contains a quotient of Gaussian primes.  
				
	We first claim that if $0< a < b$, then for sufficiently large real $x$, the open interval
	$(xa,xb)$ contains a rational prime congruent to $3 \pmod{4}$.
	Let $\pi_3(x)$ denote the number of rational primes congruent to $3 \pmod{4}$ which are $\leq x$.
	By the Prime Number Theorem for Arithmetic Progressions \cite[Thm.~4.7.4]{HardyWright},
	\begin{equation*}
		\lim_{x\to\infty} \frac{ \pi_3(x) }{ x / \log x} = \frac{1}{2},
	\end{equation*}
	whence
	\begin{align*}
		\lim_{x\to\infty} [\pi_3(xb) - \pi_3(xa)]
		&= \lim_{x\to\infty}\pi_3(xb) \left[ 1 - \frac{ \pi_3(xa) }{ \pi_3(xb) } \right] \\
		&= \lim_{x\to\infty}\pi_3(xb) \left[ 1 - \frac{ xa  \log xb }{x b \log xa } \right] \\
		&= \left(1 - \frac{a}{b}\right) \lim_{x\to\infty} \pi_3(xb)\\
		&= \infty,
	\end{align*}
	which establishes the claim.

	Next observe that
	the sector $\alpha < \arg z < \beta$ contains Gaussian primes of arbitrarily large magnitude.
	This follows from an old result of I.~Kubilyus (illustrated in Figure \ref{Figure}) which states that 
	the number of Gaussian primes $\gamma$ satisfying $0 \leq \alpha \leq \arg \gamma \leq \beta \leq 2\pi$ 
	and $|\gamma|^2 \leq u$ is
	\begin{equation}\label{eq-K}
		\frac{2}{\pi}(\beta - \alpha) \int_2^u \frac{dx}{\log x} + O\left(u \exp(-b\sqrt{\log u}) \right)
	\end{equation}
	where $b>0$ is an absolute constant \cite{Kubilyus} (see also \cite[Thms.~2,3]{HarmanLewis}).  
	
\begin{figure}[htb!]
	\begin{subfigure}{0.4\textwidth}
		\centering
		\begin{equation*}\small
			\begin{array}{|r|r|r|}
				\hline
				\rho & N & K \\
				\hline
				100 & 50 & 53 \\
				500 & 946 & 940 \\
				1,\!000 & 3,\!327 & 3,\!346 \\
				5,\!000 & 66,\!712 & 66,\!651 \\
				10,\!000 & 245,\!085 & 245,\!200 \\
				25,\!000 & 1,\!384,\!746 & 1,\!385,\!602 \\
				50,\!000 & 5,\!168,\!740 & 5,\!167,\!941 \\
				\hline
			\end{array}
		\end{equation*}
		\caption{$\frac{\pi}{24} \leq \arg z \leq \frac{2\pi}{47}$}
	\end{subfigure}
	\qquad
		\begin{subfigure}{0.4\textwidth}
		\centering
		\begin{equation*}\small
			\begin{array}{|r|r|r|}
				\hline
				\rho & N & K \\
				\hline
				1,\!000 & 0 & 5 \\
				5,\!000 & 0 & 100\\
				10,\!000 & 369 & 367 \\
				50,\!000 & 7,\!823 & 7,\!732 \\
				100,\!000 & 28,\!964 & 28,\!971 \\
				250,\!000 & 167,\!197 & 167,\!099 \\
				500,\!000 & 632,\!781& 631,\!552\\
				\hline
			\end{array}
		\end{equation*}
		\caption{$\frac{\pi}{31415} \leq \arg z \leq \frac{2\pi}{31415}$}
	\end{subfigure}

	\caption{The number $N$ of Gaussian primes in the specified sector with $|z|<\rho$, along with
	the corresponding estimate $K$ (rounded to the nearest whole number) provided by \eqref{eq-K}}
	\label{Figure}
\end{figure}

	Putting this all together, we conclude that there exists a Gaussian prime $\gamma$ in the sector 
	$\alpha < \arg z < \beta$ whose magnitude is large enough to ensure that
	\begin{equation*}
		\pi_3\left( \frac{|\gamma|}{r}\right) - \pi_3\left( \frac{|\gamma|}{R} \right) \geq 2.
	\end{equation*}
	This yields a rational prime $q \equiv 3 \pmod{4}$ such that 
	\begin{equation*}
		\frac{|\gamma|}{R} < q < \frac{|\gamma|}{r}.
	\end{equation*}
	Since $q$ is real and positive, it follows that
	$r < |\frac{\gamma}{q}| < R$ and $\alpha < \arg \frac{\gamma}{q} < \beta$ so that
	$\gamma/q$ is a quotient of Gaussian primes which belongs to the desired region \eqref{eq-Region}.
\end{proof}

\noindent\textbf{Acknowledgments}:  We thank the anonymous referees for several helpful suggestions.
This work was partially supported by National Science Foundation Grant DMS-1001614.

\bibliographystyle{monthly}
\bibliography{QGP}

\end{document}